\documentclass[11pt,fleqn]{article}

\usepackage{amsmath}
\usepackage{amssymb}
\usepackage{mathtools}
\usepackage{amsthm}
\usepackage{array}
\usepackage{graphicx}
\usepackage[square,sort,comma,numbers]{natbib}
\usepackage{relsize}
\usepackage{amsfonts}
\usepackage[all]{xy}
\usepackage{tikz-cd}
\usepackage{forest}
\forestset{
nice empty nodes/.style={
for tree={calign=fixed edge angles},
delay={where content={}{shape=coordinate,
for current and siblings={anchor=north}}{}}
},
}
\usepackage{tikz-qtree}
\usepackage{url}
\usepackage{caption}
\usepackage{subcaption}  
\usepackage{multirow}
\usepackage{tabularx}

\usepackage[plainpages=false,pdfborder={0 0 0}]{hyperref}

\newtheorem{theorem}{\textsc{Theorem}}[section]
\newtheorem{definition}[theorem]{\textsc{Definition}}
\newtheorem{example}[theorem]{\textsc{Example}}
\newtheorem{proposition}[theorem]{\textsc{Proposition}}
\newtheorem{corollary}[theorem]{\textsc{Corollary}}

\newtheorem*{remark}{\textsc{Remark}}

\newcommand{\id}{\operatorname{id}}

\newcommand{\Tot}{\operatorname{Tot}}
\newcommand{\TW}{\operatorname{Tot_{TW}}}
\newcommand{\Del}{\operatorname{Del}}
\newcommand{\Delf}{\operatorname{Del_\infty}}

\newcommand{\op}{\operatorname}
\newcommand{\de}{\partial}
\newcommand{\g}{\mathfrak{g}}
\newcommand{\gh}{\mathfrak{\tilde{g}}}
\newcommand{\ghu}{\mathfrak{\tilde{g}}(\mathcal{U})_\bullet}
\newcommand{\h}{\mathfrak{h}}
\newcommand{\f}{\mathfrak{f}}

\newcommand{\ada}{\operatorname{ad}_{a}}

\newcommand{\MC}{\operatorname{MC}}

\begin{document}

\title{Curved $L_\infty$-algebras and lifts of torsors.} 
\author{Vladimir Baranovsky and Ka Laam Chan} 

\maketitle 

\begin{abstract}
	
	Consider an extension of finite dimensional nilpotent Lie algebras 
	$0 \to \h \to \gh \to \g \to 0$ (over a field $k$ of characteristic zero) corresponding to an extension of unipotent algebraic groups $1 \to H \to \tilde{G} \to G \to 1$. 
	For a $G$-torsor $P$ on an algebraic variety $X$ over $k$, we study the problem of 
	lifting $P$ to a
	 $\widetilde{G}$-torsor $\widetilde{P}$. 
	 Fixing a trivialization of $P$ on open subsets of an affine cover, we give 
	  the Cech complex of $\h$-valued functions the structure of a curved 
	 $L_\infty$-algebra and define a curved version of the Deligne-Getzler groupoid. 
	 We show that this groupoid is isomorphic 
	  the groupoid of cocycle level $\tilde{G}$-lifts of $P$. 
\end{abstract}

\section{Introduction.} 

The purpose of this article is to give a curved version of a result due to 
Hinich, cf. \cite{hinich} on descent of Deligne groupoids, or rather its variant proved by 
Fiorenza-Manetti-Martinengo in \cite{manetti}. 
We will describe a specifc setting 
of interest. Suppose we are given a  scheme over a field $k$
of characteristic zero, an open
 affine covering 
$\mathcal{U} = \{U_i\}$ and a finite dimensional Lie algebra $\g$ with the
corresponding unipotent algebraic  group $G$. Then any $G$-torsor  $P \to G$ 
(in either Zariski of fppf topology) trivializes on each open subset $U_i$. 
For $G = \mathbb{G}_a$ this follows from the vanishing of cohomology of
the structure sheaf on each $U_i$ and in general follows by using a
filtration on $G$ with successive quotients isomorphic to $\mathbb{G}_a$. 
Therefore, given a fixed trivialization of $P$
on each $U_i$, all information about $P$ can be recovered from the transition 
functions 
$
\Phi_{ij}: U_i \cap U_j \to G
$
subject to the cocycle condition
$$
\Phi_{ik} = \Phi_{ij} \Phi_{jk}
$$
on triple intersections. As usual, a change of trivialization is described by regular 
maps $\Sigma_i: U_i \to G$ which send $\{\Phi\}_{ij}$ to an equivalent cocycle
$$
\widehat{\Phi}_{ij} = \Sigma_i^{-1} \Phi_{ij} \Sigma_j.
$$
On the other hand, the Cech complex $\mathcal{L}(\g)$ of the sheaf of $\g$-valued
functions, with respect to  $\mathcal{U}$, does not inherit the Lie 
algebra structure from $\g$ but it can be given the structure of an $L_\infty$ algebra.
With appropriate choices, the cocycle condition above can be identified with the 
Maurer-Cartan equation of $\mathcal{L}(\g)$  and equivalences of cocycles with 
arrows in the Deligne-Getzler groupoid. 

In more detail: on one hand, the $L_\infty$ structure on the Cech complex
depends on the choice of a simplicial Dupont homotopy, see \cite{getzler}.  On the other
hand, to convert the group valued cocycle condition into a statement involving the Lie algebra 
one must describe the group product in $G$ via the bracket in $\g$. This is done 
by the Campbell-Baker-Hausdorf formula, which can be modified to a family 
of formulas using the Jacobi identity on the bracket of $\g$. As it turns out, 
a choice of Dupont homotopy also fixes a CBH type formula in this family. This can be 
seen from the work of Getzler, cf. \cite{getzler} who has related the 
CBH formula to a horn lifting condition in a  Kan groupoid. The two choices
(of the $L_\infty$ structure on Cech complex and the description of group product in
terms of brackets) turn out to be coordinated so that the cocycle condition can be 
matched to the Maurer-Cartan equation.  See  \cite{manetti} for a more
general statement.

\bigskip
\noindent
We are interested in the problem of lifting a torsor across a Lie algebra extension. Thus, 
we consider an extension of nilpotent Lie algebras 
$$
0 \to \h \to \widetilde{\g} \to \g \to 0
$$
corresponding to the extension of unipotent Lie groups $H \to \widetilde{G} \to G$, 
and study the problem of lifting $P$ to a $\widetilde{G}$ torsor $\widetilde{P}$. 

This kind of problem appears in a range of geometric situations, such as 
infnitesimal extension of 
a vector bundle or a torsor from a smooth subvariety $X$ to a smooth ambient 
variety $Y$, or finding a 
deformation quantization of a vector bundle on $X$ as a module over a fixed
Zariski sheaf $\mathcal{O}_h$ of quantized functions on $X$. For full power
appropriate to such applications one should rather work with groups and 
Lie algebras which are infinite-dimensional proalgebraic groups with the finite 
dimensional reductive Levi part. In this paper, however, we restrict to the simple finite
dimensional unipotent/nilpotent setting, just observing that the case of the
nontrivial redutive part $G_L$ would be resolved by working with $G_L$-invariant
objects on appropriate $G_L$-torsor oveer $X$. See e.g. Section 5.1 in 
\cite{NT} for an example. 

To formulate our result we also choose a splitting $\widetilde{\g} \simeq \g 
\oplus \h$ as vector spaces. Since each Lie algebra can be identified with its
 group as an algebraic variety, we also get a splitting of algebraic
varieties $\widetilde{G} \simeq G \times H$ (the group operation is 
compatible with the embedding of $H$ and the projection to $G$). 

\begin{theorem}
	\label{main}
	For a fixed cocycle definiing $P$ and the choices made above, there exists
	a curved complete $L_\infty$ algebra structure on the Cech complex 
	$\mathcal{L}(\h)$ of $\h$-valued functions, such that  its curved Maurer-Cartan 
	elements are in bijective correspondence 
	with $\widetilde{G}$-valued lifts of the cocycle defining $P$, 
	and homotopy equivalences between 
	different solutions are in bijective correspondence with $H$-valued changes of
	trivialization. 
\end{theorem} 

We note that curved Maurer-Cartan solutions behave most reasonably in a complete
filtered setting, and while we were able to define the Deligne-Getzler groupoid in 
our particular situation, we are not aware of any general appropriate construction 
that implies our case. 

The structure of this paper is as follows. In Section 2 we recall definitions and fix notation 
related to $L_\infty$ algebras, Maurer-Cartan solutions and their equivalence, and formulate
results related to homotopy transfer of structure under a contraction. In Section 3
we discuss the construction of an $\infty$-groupoid associated with an $L_\infty$
algebra (which we assume to be  non-negatively graded). We also recall the relationship between CBH and horn filling. Section 4 
discusses the general theorem of Hinich on descent of Deligne groupoids
(for a semicosimplicial DGLA) and its more specific version for semicosimplicial 
Lie algebras, due to Fiorenza-Manetti-Martinengo. We explain how the latter result 
relates to deformation of $G$-torsors. Finally, in Section 5-7 we extend it to the curved
setting and compare curved Maurer-Cartan equations to 
lifts of cocycles across extensions. Definitons and results related to curved $L_\infty$
algebras are collected in Section 8 (appendix).

\section{Complete $L_\infty$ algebras and Maurer-Cartan 
	equations.}

\subsection{$L_\infty$ Algebras and $L_\infty$ morphisms}

We follow the notation used in \cite{bthesis}.

\begin{definition}
	An \(\mathbf{L_\infty}\) \textbf{structure} on a graded vector space \(V\) is  a codifferential \(Q\) of degree 1 on the (reduced) symmetric coalgebra \(\overline{S(V[1])} 
	= \bigoplus_{i \geq 1} V[1]^{\odot i} \) .
	The pair $(V, Q)$ is called an \(\mathbf{L_\infty}\) \textbf{algebra}.
\end{definition}

Note that the codifferential \(Q\) is determined by \(Q^1: \overline{S(V[1])} \to V\), see Corollary VIII.34 in \cite{manettilec}.  If we break the map apart, we will get maps \(q_1 = Q^1_1: V \to V\), \(q_2 = Q^1_2: V\odot V \to V\), \(q_3 = Q^1_3: V \odot V \odot V \to V\), etc. 
The fact that $Q^2 = 0$ implies that the $q_i$'s have to satisfy a series of equations,  called the generalized Jacobi identities. In particular $q_1^2=0$.

\begin{definition}
	An \textbf{\(\mathbf{L_\infty}\) morphism} of 
	$L_\infty$ algebras \(F: (V,Q) \to (W,R)\)  is a morphism of dg-coalgebras \(G: \overline{S(V[1])} \to \overline{S(W[1])}\) which is given by a family of degree zero maps $g_i = G^1_i:V[1]^{\odot i} \to W[1]$, $i\geq 1$, such that \(G: \overline{S(V[1])} \to \overline{S(W[1])}\) commutes with $Q$ and $R$.
 \(F\) is  a \textbf{strict morphism} if $g_i = 0$ for $i \geq 2$.  
	\(F\)  is an \(\mathbf{L_\infty}\) \textbf{quasi-isomorphism}
	if \(g_1:(V[1],q_1) \to (W[1],r_1)\) is a quasi-isomorphism of 
	underlying complexes.
\end{definition}

In this paper, we work with algebras that have a complete filtration.  Completeness is not needed for the homotopy transfer of structure theorem, but it is essential in the proof of the formal Kuranishi theorem.  

\begin{definition}
	A \textbf{complete graded space} is a graded space $V$ equipped with a descending filtration $F^\bullet V$,
	\[V = F^1V \supset \dots \supset F^p V \supset \dots\]
	such that $V$ is complete in the induced topology, i.e. the natural $V \to \underleftarrow{\operatorname{lim}}V/F^{\bullet}V$ is an isomorphism of graded spaces.  Given complete graded spaces $(W, F^{\bullet}W)$ and $(V, F^{\bullet}V)$, a \textbf{continuous} map of graded spaces is a map $g: W \to V$ such that $g(F^pW) \subset F^pV$ for all $p \geq 1$. A \textbf{complete dg space} is a complete graded space $(V, F^{\bullet}V)$ with a continuous differential $d$.
\end{definition}

\begin{definition}
	
	A \textbf{complete} $\mathbf{L_\infty}$ \textbf{algebra} is a complete graded space $(V, F^{\bullet}V)$ with an $L_\infty$ structure $Q$ on $V$ such that $q_i$'s are continuous, i.e. $q_i(F^{p_1}V[1] \odot \dots \odot F^{p_i}V[1]) \subset F^{p_1 + \dots + p_i}V[1]$, for all $i, p_1, \dots, p_i \geq 1$.
	A \textbf{continuous} $\mathbf{L_\infty}$ \textbf{morphism} $G:(W, F^{\bullet}W, R) \to (V, F^{\bullet}V, Q)$ between complete \(L_\infty\) algebras  is an $L_\infty$ morphism $G: (W,R) \to (V,Q)$ such that $g_i$'s are continuous, i.e. $g_i(F^{p_1}W[1] \odot \dots \odot F^{p_i}W[1]) \subset F^{p_1 + \dots + p_i}V[1]$, for all $i, p_1, \dots, p_i \geq 1$.
\end{definition}

\begin{remark}
	From now on we assume completeness of algebras and morphisms, 
	unless otherwise specified. 
\end{remark}

\begin{definition}
	The \textbf{curvature} of a complete $L_\infty$ algebra $(V, F^\bullet V, Q)$ is the map 
	\[ x \mapsto 
	\mathcal{R}_V(x) = \displaystyle\sum_{i \geq 1} \frac{1}{i!} q_i (\overbrace{x \odot \dots \odot x}^i):  V^1 \to V^2\]
	Note that the infinite sum above converges because $(V,F^\bullet V,Q)$ is complete.
	The \textbf{Maurer Cartan set} of a complete $L_\infty$ algebra $(V, F^\bullet V, Q)$ is the set 
	\[\op{MC}(V):= \{x \in V^1 \mbox{ s.t. } \mathcal{R}_V(x) = 0\}.\]
\end{definition}
A continuous $L_\infty$ morphism $G$ induces a map, see e.g. \cite{bthesis}:
 $$
 G_*:\op{MC}(W) \to \op{MC}(V), 
x\mapsto G_\ast(x) = \sum_{i\geq 1}\frac{1}{i!}g_i(\overbrace{x \odot \dots \odot x}^i) 
$$
Gauge equivalence used with a Maurer-Cartan solutions of a positively graded
DGLAs is not defined on Maurer Cartan solutions of an $L_\infty$ algebra $L$ as $L^0$ is not a Lie algebra in general.  Instead, we will use the following equivalence for Maurer Cartan solutions,
given in terms of the induced $L_\infty$ algebra $V\otimes_k k[s,ds]$ where 
$k[s,ds]$ is the graded commutative algebra on a degree 0 variable $s$, 
and a degree 1 variable $ds$ subject to $d(s) = ds$:
\begin{definition}
	Two Maurer Cartan solutions $a,a' \in \MC(V)$ are \textbf{(homotopy) equivalent} if there exist $z \in \MC(V\otimes_k k[s,ds])$ such that
	\[z|_{s=0} = a, \qquad \qquad z|_{s=1} = a'\]
	where the evaluation map is given by $\op{Eval}_{s = s_0}: V \otimes_k k[s,ds] \to V$ 
	\[\op{Eval}_{s=s_0}(x(s)+y(s)ds) = x(s_0)\]
\end{definition}

\subsection{Homotopy Transfer and Formal Kuranishi Theorem}
In this section we will review the homotopy transfer of structure theorem and observe how the Maurer Cartan set behaves under homotopy transfer.

\begin{definition} A \textbf{complete contraction} 
	\[
	\begin{tikzcd}
		W \arrow[r,shift left, "f"] \arrow[r, leftarrow, shift right, "g", swap] & V \arrow[loop right, leftarrow]{l}{K}
	\end{tikzcd}
	\]
	is a complete dg space $(V,F^\bullet V,d_V)$ and a dg space $(W,d_W)$, together with dg morphisms $f:(W,d_W)\to (V,d_V)$, $g:(V,d_V)\to(W,d_W)$ and a contracting (degree minus one) homotopy $K:V\to V$, such that 
	\begin{itemize}
		\item  $gf=\id_W$ and       $Kd_V+d_VK=fg-\id_V$
		\item $K$ satisfies the side conditions $Kf = K^2 = gK= 0$
		\item $K$ and $fg$ are continuous with respect to the filtration $F^\bullet V$ on $V$.
	\end{itemize}
	Then $W$ is equipped with the induced filtration $F^pW = f^{-1}(F^pV)$ such that  $f,g$ are continuous morphisms.
\end{definition}

We now state the homotopy transfer of structure theorem, see e.g. \cite{barticle}.  
It can be used for constructing $L_\infty$ structures on a dg vector space. 

\begin{theorem} 
	Given a complete contraction 
	$\begin{tikzcd}
		W[1] \arrow[r,shift left, "f_1"] \arrow[r, leftarrow, shift right, "g_1", swap] & V[1] \arrow[loop right, leftarrow]{l}{K}
	\end{tikzcd}$
	and a complete $L_\infty$ algebra structure $Q$ on $(V,F^\bullet V)$ with linear part $q_1=d_{V[1]}$, there is an induced complete $L_\infty$ algebra structure $R$ on $(W,F^\bullet W)$ with linear part $r_1=d_{W[1]}$, together with continuous $L_\infty$ morphisms $f:(W,R)\to(V,Q)$, $g:(V,Q)\to (W,R)$ with linear parts $f_1$, $g_1$ respectively. Denoting by $f^k_i$ the composition $W[1]^{\odot i}\hookrightarrow\overline{S(W[1])}\xrightarrow{f}\overline{S(V[1])}\twoheadrightarrow V[1]^{\odot k}$, $f$ and $R$ are determined recursively by 
	\[ f_i = \sum_{k=2}^i Kq_k f^k_i \qquad
	r_i = \sum_{k=2}^i g_1q_k f^k_i \qquad \mbox{for $i\geq2$}. \]	
	We denote by $K_i^{\Sigma}:V[1]^{\odot i}\to V[1]^{\odot i}$ the degree minus one map defined by 
	\[ K^\Sigma_i(v_1\odot\cdots\odot v_i) =\]
	\[\frac{1}{i!} \sum_{\sigma\in S_i, j = 0, \ldots, i} \pm f_1g_1(v_{\sigma(1)})\odot\cdots\odot f_1g_1(v_{\sigma(j-1)})\odot K(v_{\sigma(j)})\odot v_{\sigma(j+1)}\odot\cdots\odot v_{\sigma(i)}, \]
	where $\pm$ is the appropriate Koszul sign (taking into account that $|K|=-1$).  Denoting by $Q^k_i$ the composition $V[1]^{\odot i}\hookrightarrow\overline{S(V[1])}\xrightarrow{Q}\overline{S(V[1])}\twoheadrightarrow V[1]^{\odot k}$, the $L_\infty$ morphism $g$ is determined recursively by
	\[ g_i = \sum_{k=1}^{i-1} g_k Q^k_i K^\Sigma_i\qquad\mbox{for $i\geq2$}.  \]
\end{theorem}

\medskip
\noindent
The following theorem is essentially due to Getzler \cite{getzler} explains what happens to the 
Maurer Cartan set under the homotopy transfer. It is stated in this form by Bandiera in
\cite{barticle}, where it is called the  formal Kuranishi theorem.  

\begin{theorem}
	Under the hypothesis of the homotopy transfer of structure theorem, the correspondence
	\[\rho:\op{MC}(V)\to \op{MC}(W)\times K(V^1):x\to(\op{MC}(g)(x),K(x))\]
	is bijective. The inverse $\rho^{-1}$ admits the following recursive construction: for $y\in \op{MC}(W)$ and $K(v)\in K(V^1)$, define $x_n\in V^1$, $n\geq 0$, by $x_{0}=0$ and 
	\[x_{n+1}=f_1(y)-q_1K(v)+\sum_{i\geq2}\frac{1}{i!}\left(Kq_i-f_1g_i\right)\left(x_n^{\odot i}\right).\]
	This sequence converges (with respect to the complete topology on $V$) to a well defined $x\in V^1$, and we have  $\rho^{-1}(y,K(v))=x$. Finally, $\rho^{-1}(-,0)= \op{MC}(f): \op{MC}(W)\to \op{MC}(V)$ is a bijection $\op{MC}(W) \to \op{Ker}\,K\bigcap \op{MC}(V)$, whose inverse is the restriction of $g_1$.
\end{theorem}

\section{Deligne-Getzler $\infty$-Groupoid}
For a nilpotent Lie algebra $\g$ the exponential map gives a bijection between $\g$ and 
its unipotent group $G$ (and the Baker-Campbell-Hausdorff gives the group product in terms of brackets on $\g$).   In the 
$L_\infty$ case, $L^0$ is not a Lie algebra as the bracket does not satisfy the Jacobi identity.  We need an object that generalizes the Lie group $G$. 
 It turns out that the natural object to consider will be $\infty$-groupoids, which are simplicial 
 sets satisfying some additional conditions.  
In  \cite{getzler}, Getzler  explains how to integrate a nilpotent $L_\infty$ algebra to an $\infty$-groupoid, which generalizes the way a nilpotent Lie algebra integrates to its exponential group.  General Baker-Campbell-Hausdorff product can then be seen as an arrow filing a horn of a Kan complex.  A first model for such $\infty$-groupoid, $\op{MC}_\infty(L):= \op{MC}(\Omega^* (\Delta_\bullet ; L))$, was introduced by Sullivan and studied in depth by Hinich \cite{hinich}.  The problem with $\op{MC}_\infty(L)$ is that 
it is quite large, e.g. larger than the Cech complex in the geometric setting (and it is not a $\infty$-groupoid in a strict sense \cite{bthesis}).  When $\g$ is a nilpotent Lie algebra, the nerve $\mathcal{N}(e^{\g})$ is only a deformation retraction of $\op{MC}_\infty(\g)$.  Getzler introduced a smaller model $\gamma_\bullet$ homotopy equivalent to $\op{MC}_\infty(L)$ as a Kan complex.  Bandiera rewrites $\gamma_\bullet$ as $\Del_\infty(L):= \op{MC}(C^*(\Delta_\bullet ; L))$ using the formal Kuranishi theorem \cite{barticle}, which is the notation we are going to use.

\subsection{Groupoids and $\infty$-groupoids}

We recall the standard simplicial definitons. 
Let \(\Delta\) be the simplex category with ordinals \([n] = (0 < 1 < \dots < n)\) as objects and non decreasing maps as morphisms.  It is generated by the injective face maps \(d_k:[n-1] \to [n]\), \(0\leq k \leq n\),(the image of $d_k$
does not contain $k$); 
and the surjective degeneracy maps \(s_k:[n] \to [n-1]\), \(0 \leq k \leq n-1\)
(the value $k$ is repeated twice for $s_k$).  

\begin{definition}
	A simplicial set \(X\) is a contravariant functor from \(\Delta\) to the category of sets.  This gives us a sequence of sets \(X_n = X([n])\) indexed by the natural numbers \(n\in \{0, 1, 2, \dots\}\), and the maps 
	\begin{align*}
		\delta_k &= X(d_k) : X_n \to X_{n-1}, & 0 \leq k \leq n\\
		\sigma_k &= X(s_k) : X_{n-1} \to X_n, & 0 \leq k \leq n
	\end{align*}
	satisfying the simplicial identities, cf. \cite{May}. 
\end{definition}

In other words, $X$ is a functor $\Delta^{op} \to Sets$. Similarly, we can 
define simplicial objects in any category.

\begin{definition}
	Let $\Delta_n = \Delta(-,[n]) \in \op{SSet}$ be the standard simplicial $n$-simplex in $\op{SSet}$.  For $0 \leq i \leq n$, let $\Lambda^i_n \subset \Delta_n$ be simplicial set 
	corresponding to the union of the faces $d_k[\Delta_{n-1}] \subset \Delta_n$, $k \neq i$.  An \textbf{n-horn} in $X$ is a simplicial map from $\Lambda^i_n$ to $X$.
	A simplicial object $X$ satisfies the \textbf{Kan condition} if any morphism of n-horn can be extended to a simplicial morphism $\Delta_n \to X$.  Such $X$ is called a \textbf{Kan complex}.
\end{definition}

We will now introduce the nerve functor, which associate each groupoid (group) a corresponding simplicial set.

\begin{definition}
	Given a groupoid $G$ (a category where all morphisms are invertible), the \textbf{nerve} $\mathcal{N}(G)$ of G is a simplicial set whose 0-simplices are objects of $G$, 1-simplices morphisms of $G$, and $n$-simplices are $n$-tuples of composable morphisms of $G$, i.e.
	\[x_0 \overset{f_1}{\to} \dots \overset{f_n}{\to} x_n\]
	where $x_i$ is an object in $G$ and the $f_i: x_{i-1} \to x_i$ is a morphism from $x_{i-1}$ to $x_i$.
	The face maps $d_i: \mathcal{N}(G)_k \to \mathcal{N}(G)_{k-1}$
	are given by composition of morphisms at the $i$-th object. 
	The degeneracy maps $s_i: \mathcal{N}(G)_k \to \mathcal{N}(G)_{k+1}$
	are given by inserting identity morphism at the object $x_i$.
\end{definition}

\begin{proposition} \cite{bthesis} 
	Given a groupoid $G$, the nerve $\mathcal{N}(G)$ is a Kan complex.
\end{proposition}
Thus Kan complexes give us a generalization for groupoids (groups). We will 
further narrow down to simplicial sets called $\infty$-groupoids, which 
Kan complexes with an additional class of \textit{thin elements} such that every 
horn has a unique thin filler, see Definition 2.5 in \cite{getzler}. 

\begin{definition}
	Two parallel 1-simplices $f$ and $g$ of a Kan complex $X$ are \textbf{homotopic} if and only if there exist a 2-simplex in $X$ of either of the following form
	
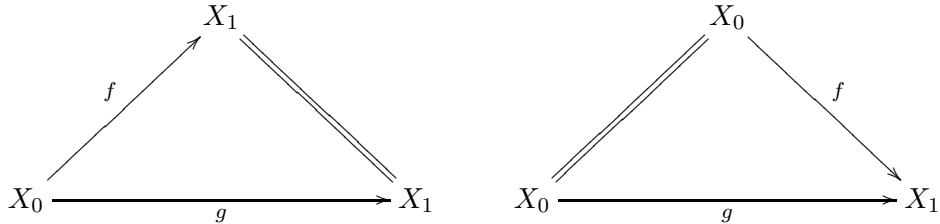
\[ \xymatrix{ & & X_1\ar@{=}[rrdd] & & & & & X_0\ar[rrdd]^f & &\\ & & & & & & & & & & & &\\ X_0\ar[uurr]^f \ar[rrrr]_g & & & & X_1 &  X_0\ar@{=}[uurr]\ar[rrrr]_g & & & &X_1}\]
	This defines an equivalence relation on the 1-simplices of $X$ \cite{infinite}.
\end{definition}

The left adjoint to the nerve functor, $\mathcal{N}: \op{Grpd} \to \op{Kan}$, which takes Kan complexes back to groupoids is called the fundamental groupoid functor.

\begin{definition}
	Given a Kan complex $X$, the \textbf{fundamental groupoid}, $\pi_{\leq 1} X$, is the groupoid with the following properties:
	\begin{itemize}
		\item the set of objects are 0-simplices in $X$
		\item the morphisms are homotopy classes of 1-simplices in $X$
		\item the identity morphism of $x \in X_0$ is represented by the degenerate 1-simplex $s_0(x)$
		\item a composition relation $h = g \circ f$ in $\pi_{\leq 1} X$ if and only if for any choices of 1-simplices representing these morphisms, there exist a 2-simplex in $X$ with boundary
		
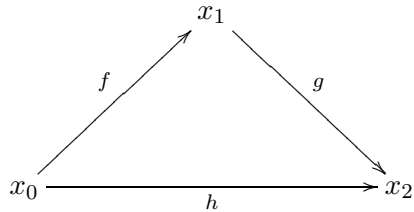
\[
		\xymatrix{ & & x_1\ar[rrdd]^g & &\\ & & & & &\\ x_0 \ar[uurr]^f \ar[rrrr]_h & & & & x_2}
		\]
	\end{itemize}
\end{definition}

The fundamental groupoid of a Kan complex mimics the fundamental groupoid of a topological space.  The following proposition tells us that the $\pi_{\leq 1}$ functor preserves the homotopy relation.

\begin{proposition}
	\label{isom-groupoids} 
	If $X$ and $Y$ are homotopy equivalent Kan complexes, then $\pi_{\leq 1} X$ and $\pi_{\leq 1} Y$ are equivalent as groupoids. If $X$ and $Y$ are homotopy equivalent Kan complexes and that $\pi_{\leq 1}X$ and $\pi_{\leq 1}Y$ have the same set of objects, then $\pi_{\leq 1} X$ and $\pi_{\leq 1} Y$ are isomorphic groupoids.
\end{proposition}
\begin{proof}
	See \cite{infinite} for the proof of the first part. For the second, observe that equivalent groupoids with the same set of objects are isomorphic.
\end{proof}

\subsection{Deligne-Getzler $\infty$-groupoids}
In this subsection we  first introduce two important complexes and from them construct the Deligne-Getzler $\infty$-groupoid that gives us the general Baker-Campbell-Hausdorff product of a DGLA or an $L_\infty$ algebra.

\begin{definition}
	For  $n\geq 0$, the \textbf{differential graded commutative algebra of polynomial differential forms on the} $\mathbf{n}$\textbf{-simplex} $\mathbf{\Delta_n}$ is:
	\[\Omega^*_n = \frac{k[t_0, \dots, t_n, dt_0, \dots, dt_n]}{(\sum t_i - 1, \sum dt_i)}.\]
	where the differential is induced by the usual differential for differential forms that sends $t_i \to dt_i$.  Notice that $\Omega^\ast_\bullet$ has a natural structure of simplicial dg commutative algebra.  The face map $\partial_i$
	annihilates $t_i$ and $dt_i$ 
	and the degeneracy map $s_i$ sends $t_i$, $dt_i$ to $t_i + t_{i+1}$ and
	$dt_i + dt_{i+1}$, respectively. 
	
	Given a simplicial set $X$, the \textbf{space of polynomial $\mathbf{l}$-forms} on $X$ is $\Omega^l(X):= \op{SSet}(X,\Omega^l)$, i.e. the simplicial set morphisms from $X$ to $\Omega^l$, and $\Omega^*(X):= \oplus_{l\geq 0}\Omega^l (X)$.  In particular, when $X$ is $\Delta_\bullet$, we have $\Omega^*(\Delta_\bullet) = \Omega^*$.
	
	Given a dg vector space $L$, set $\Omega^*(X;L) = \Omega^*(X) \otimes L$ \textbf{the complex of polynomial differential forms on} $\mathbf{X}$ \textbf{with coefficients in} $\mathbf{L}$
\end{definition}

\begin{definition}
	The \textbf{complex of non-degenerate simplicial} $\mathbf{k}$\textbf{-cochains on }$\mathbf{X}$ is $C^*(X):=C^*(X; k)= \oplus_{l\geq 0} C^l(X)$ where $C^l(X)$ is the space of $k$-valued $l$-cochains $\alpha: X_l \to k; \sigma \mapsto \alpha_\sigma$ on $X$ vanishing on degenerate simplices.  The differential is given by
	\[
	d\alpha(\sigma)
	=(d\alpha)_{\sigma}
	=\displaystyle \sum_{i=0}^{k+1} (-1)^i \alpha_{\partial_i \sigma}
	\]
	where $\partial_i:X_{k+1} \to X_k, i = 0, \dots, k+1,$ are the face maps of $X$.
	
	Given a dg vecttor space $L$, set $C^*(X;L) = C^*(X) \otimes L$ \textbf{the complex of non-degenerate simplicial cochains on} $\mathbf{X}$ \textbf{with coefficients in} $\mathbf{L}$.
\end{definition}

Note that $\Omega^*(X;L)$ inherits an algebraic structure of  $L$ (such as Lie or graded commutative)
 as $\Omega^*(X)$ is graded commutative.  The complex $\Omega^*(X;L)$ does not inherit a complete structure in general, but we can replace $\Omega^*(X;L)$ with its completion $\widehat{\Omega}^*(X;L) := \underleftarrow{\operatorname{lim}} \Omega^*(X; L/F^p L)$ which will also have the same algebraic structure as $L$.  

On the other hand, for a complete  $L$ the complex $C^*(X; L)$ is complete with respect to the filtration $F^p C^*(X;L) = C^*(X; F^p L)$, but in general $C^*(X; L)$ does not inherit an algebraic
 structure of $L$.  But a standard contraction from $\Omega^*(X;L)$ to $C^*(X;L)$ (and thus from $\hat{\Omega}^*(X;L)$ to $C^*(X;L)$) can be used to transfer a \textit{homotopy} version of 
structure (Lie, graded commutative, etc.)  to  $C^*(X;L)$ using the structure on $\hat{\Omega}^*(X;L)$.

\begin{theorem}[Getzler, \cite{getzler}]
	There is a standard contraction from $\hat{\Omega}^*(X;L)$ to $C^*(X;L)$ given by integrating forms over simplices in one direction, inclusion of Whitney's elementary forms in the other direction, and Dupont homotopy as the contracting homotopy.
\end{theorem}

In particular when $L$ is a complete $L_\infty$ algebra, so is $\widehat{\Omega}^*(X;L)$, and by 
homotopy transfer $C^*(X;L)$  is a complete $L_\infty$ algebra.
We now define the Deligne-Getzler $\infty$-groupoid of a complete DGLA ($L_\infty$ algebra) $L$.  Denote
\[
\xymatrix{ \Delta_\bullet : {\Delta_0}
	\ar@<2pt>[r]\ar@<-2pt>[r] & { \Delta_1}
	\ar@<4pt>[r] \ar[r] \ar@<-4pt>[r] & { \Delta_2}
	\ar@<6pt>[r] \ar@<2pt>[r] \ar@<-2pt>[r] \ar@<-6pt>[r]&
	\cdots}
\]
the standard cosimplicial simplex in $\op{SSet}^{\Delta}$.  

\begin{definition}
	Given a complete $L_\infty$ algebra $L$, the \textbf{Deligne-Getzler $\mathbf{\infty}$-groupoid} of $L$ is the simplicial set $\Del_\infty(L)_n:=\op{MC}(C^*(\Delta_n;L))$ of Maurer-Cartan cochains
	with coefficients in $L$, its face an degeneracy maps induced from $\Delta_\bullet$ .  Further 
	$\MC_\infty(L)$ is the simplicial set with $\MC_{\infty} (L)_n  := \MC(\Omega^*(\Delta_n; L))$
	 with face and degeneracy maps induced  from $\Omega_\bullet$.
\end{definition}

The properties of these two simplicial sets are summarized below 

\begin{proposition}  \label{groupoid-properties} 
	The following hold: 

	(a) $\MC_\infty(L)$ and $Del_\infty(L)$ are homotopy equivalent $\infty$-groupoids. 
	
	(b) $\pi_{\leq 1}\MC_\infty(L)$ is isomorphic to $\pi_{\leq 1}\Delf(L)$ as groupoids.
	
	(c) For a non negatively graded nilpotent $L_\infty$ algebra $L$,  $\Delf(L)$ is isomorphic to 
	the nerve of  $Del (L) := \pi_{\leq 1}\Delf(L)^{op} \simeq \pi_{\leq 1}\MC_\infty(L)
	^{op}$. Further, in this case morphisms in $\pi_{\leq 1}\MC_\infty(L)$, are in bijection with 1-simplices of $\Delf(L)$.
\end{proposition} 

\begin{proof}
	For part (a), the Kan property for  $\MC_\infty$ is proved in \cite{hinich}, while the Kan property
	for  $\Delf$ and homotopy equivalence
	 in  Corollary 5.9 of \cite{getzler} (the latter paper actually deals with an isomorphic
	simplical set denoted there by $\gamma_\bullet(L)$). 
	morphism on $\pi_{\leq 1}$ in part (b) follows by Proposition \ref{isom-groupoids}. 
	In part (c), isomorphism with the nerve is a particular case of Theorem 5.4
	of  \cite{getzler}. 
	For statement about morphisms, note that under the assumption on grading,  morphisms in $\Del^{\op{op}}(L)$ are given by 1-simplices of $\Delf(L)$ as $\Delf(L) = \mathcal{N}(\Del^{\op{op}}(L))$.
\end{proof}

\subsection{Baker-Campbell-Hausdorff and Horn Filling}
In this section, we assume that an nilpotent $L_\infty$ algebra $L$ is concentrated
in non-negative degrees and, following \cite{bthesis},
 we relate horn filling in $\Del_\infty(L)$ with 
 Baker-Campbell-Hausdorff product on $L^0$.  We will start by looking at a complete contraction given by Bandiera in \cite{barticle}.
Let $L$ be a complete $L_\infty$ algebra.  For $i=0,\ldots,n$, we define a homotopy $h^i:C^*(\Delta_n;L)\to C^{*-1}(\Delta_{n};L)$ by writing for $0\leq i_0<\cdots<i_k\leq n$:
\begin{equation*}  h^i(\alpha)_{i_0\cdots i_k}=\left\{\begin{array}{ll} 0 &\mbox{if $i\in\{i_0,\cdots,i_k\}$}\\ (-1)^j\alpha_{i_0\cdots i_{j-1}ii_{j}\cdots i_{k}}&\mbox{if $i_{j-1}<i<i_j$}\end{array}\right.
\end{equation*}
where $\beta_{i_0\cdots i_k}\in L^{i-k}$, $0\leq i_0<\cdots< i_k\leq n$ is the evaluation of $\beta\in C^i(\Delta_n;L)$ on the $k$-simplex of $\Delta_n$ spanned by the vertices $i_0,\ldots,i_k$. Denote $e_i:\Delta_0\to\Delta_n$ the inclusion of the $i$-th vertex  and by $\pi:\Delta_n\to\Delta_0$ the final morphism. The above  $h^i$ gives a homotopy on the complete contraction
\[
\begin{tikzcd}
	L = C^*(\Delta_0;L) \arrow[r,shift left, "\pi^*"] \arrow[r, leftarrow, shift right, "e_i^*"'] &C^*(\Delta_n;L) \arrow[loop right, leftarrow]{l}{h^i}
\end{tikzcd}
\]
If $\de_i:\Delta_{n-1}\to\Delta_n$ is the inclusion of the $i$-th face, then $\partial_i^\ast$ sends $h^i(C^{1}(\Delta_n;L))$ isomorphically to $C^{0}(\Delta_{n-1};L)$.
The formal Kuranishi theorem by with $W = L$, $V = C^*(\Delta_{n};L)$ and $K = h^i$, gives the following proposition \cite{barticle}:

\begin{proposition}[\cite{bthesis}]
	For all $i=0,\ldots,n$, the correspondence
	\[\rho^i:\Del_\infty(L)_n\to \op{MC}(L)\times h^i(C^1(\Delta_n;L)):\alpha\to (e_i^\ast(\alpha),h^i(\alpha))\]
	is bijective.
\end{proposition}

In other words, if we fix a Maurer Cartan element in $L$ and a cochain in $h^i(C^1(\Delta_n;L))$, we can recover the unique cochain in $\Del_\infty(L)_n$ by the recursive formula of the formal Kuranishi theorem.

This allows us to recover (slightly generalized) gauge action of $L^0$ on Maurer-Cartan solutions. 
Indeed, take $x \in \op{MC}(L) \subset L^1$ and $a \in L^0$ and consider the pair 
$(x, \alpha)$ where $\alpha \in h^0(C^1(\Delta_1; L) \subset C^0(\Delta_0; L)$ takes 
the value $a$ on the unique non-degenerate 0-simplex of $\Delta_0$. Then $z = (\rho^0)^{-1} (x, \alpha)
\in Del_\infty(L)_1$ and $\partial^*_0(z)$ is the Maurer-Cartan element that we denote by 
$a \cdot x \in L^1$. By Section 5.2 in \cite{bthesis} this gives the usual gauge action of $L^0$
on Maurer-Cartan elements when $L$ is a DGLA.

Now take any $x \in \op{MC}(L)$ (say $x=0$) and $a, b \in L^0$.  We want to get the Baker-Campbell-Hausdorff product of $a$ and $b$ through the horn filling in $\Del_\infty(L)_2$.   Consider the following 2-horn:  Put $x$ on the [1] vertex, $a$ on the [01] edge, $b$ on the [12] edge, and 0 on [012]
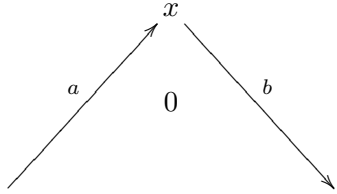
\[ \xymatrix{ & & x\ar[rrdd]^b & &\\ & & 0 & &\\ \ar[uurr]^a & & & &}\]
Now consider the 1-cochain $\alpha'$ in $h^1(C^1(\Delta_2;L)) \subset C^0(\Delta_2, L)$ with the only nonzero 
values on non-degenerate simplices in $\Delta_2$ given by 
\[\alpha'([0]) = a, \qquad \alpha'([2]) = b.\]

We can then apply the recursive formula from the formal Kuranishi theorem and get an unique cochain $\alpha \in \op{MC}(C^*(\Delta_2;L)) = \Del_\infty(L)_2$.  The Baker-Campbell-Hausdorff product $\rho^x_2(-)$ between the morphism $a$ and $b$ is then defined by evaluating $\alpha$ on the face $\partial_1 \Delta_2$ opposite to the vertex [1].
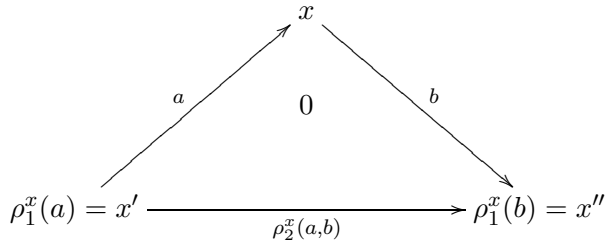
\[ \xymatrix{& & x\ar[rrdd]^b & &\\
	   & & 0 & &\\ 
	 \rho^x_1(a)=x'\ar[uurr]^a\ar[rrrr]_{\rho_2^x(a,b)} & & & &\rho^x_1(b)= x''}\]

Then  $\rho^x_2(a,b) = a*b$, the usual Baker-Campbell-Hausdorff product,
  by Proposition 5.2.36 in \cite{bthesis}.
In general we have higher general Baker-Campbell-Hausdorff products obtained 
by filling in higher dimensional horns but in this paper 
we only deal with  the  products between elements in $L^0$.

\section{Descent of Deligne groupoids}
Given a nilpotent Lie algebra $\g$ with the unipotent group $G = exp(\g)$
 and a $G$ torsor $P$ over $X$, the result of Hinich in 
 \cite{hinich}
 allows us to use combinatorial tools to study formal deformations of  $P$.  
 One can either work with the Thom-Whitney complex constructed from an
 affine covering of $X$, or a quasi-isomorphic $L_\infty$ algebra constructed
 on the underlying Cech complex. 
 
 This is a special case of the $L_\infty$ structure for semicosimplicial 
 Lie algebras (in degree 0) is studied by Fiorenza, Manetti, and Martinengo.  They show in \cite{manetti} that the solutions to the deformation equation (i.e. cocycle condition on transition functions) are exactly the Maurer Cartan solutions of the $L_\infty$  Cech complex and the equivalences of deformations are exactly the the equivalence of Maurer Cartan solutions.
We use Bandiera's reformulation of the result in \cite{manetti} using Deligne-Getzler $\infty$-groupoids.

\subsection{Semicosimplicial DGLAs, totalization and homotopy limit}
 
\begin{definition} A semicosimplicial differential graded Lie algebra
	$L_\bullet$ is  a covariant functor \(\underrightarrow{\Delta} \to DGLAs\) from the category \(\underrightarrow{\Delta}\), whose objects are finite ordinals and morphisms are order-preserving injective maps, to the category of DGLAs.  In other words, $L_\bullet$ is a diagram 
	\[
	\xymatrix{ L_\bullet : {L_0}
		\ar@<2pt>[r]\ar@<-2pt>[r] & { L_1}
		\ar@<4pt>[r] \ar[r] \ar@<-4pt>[r] & { L_2}
		\ar@<6pt>[r] \ar@<2pt>[r] \ar@<-2pt>[r] \ar@<-6pt>[r]&
		\cdots}
	\]
	where each \(L_i\) is a DGLA and the DGLA morphisms
		$\partial_{k,i}: L_{i-1} \to L_i,  k = 0, \dots, i,$ satisfy  $\partial_{k+1,i+1}\partial_{l,i} = \partial_{l,i+1}\partial_{k,i}$ for any $k \geq l$.
\end{definition}

A very important source of such objects comes from a space $X$ with an 
open covering $\{U_i\}$ and a sheaf of Lie algebras $\g$ on $U$.

\begin{definition}
	Let $X$ be a scheme over $k$, $\g$  a sheaf of Lie algebras on $X$, and $\mathcal{U}$ an affine open cover of $X$. Define the semicosimplicial Lie algebra
	\[
	\xymatrix{ \g(\mathcal{U})_\bullet : {\g(\mathcal{U})_0}
		\ar@<2pt>[r]\ar@<-2pt>[r] & {\g(\mathcal{U})_1}
		\ar@<4pt>[r] \ar[r] \ar@<-4pt>[r] & {\g(\mathcal{U})_2}
		\ar@<6pt>[r] \ar@<2pt>[r] \ar@<-2pt>[r] \ar@<-6pt>[r]&
		\cdots}
	\]
	where $\g(\mathcal{U})_q:=\prod_{i_0< \dots < i_q}\g(U_{i_0 \dots i_q})$ and arrows 
	are induced by restriction.  
\end{definition}

The usual Cech complex is $\g(\mathcal{U}) = \bigoplus_q \g(\mathcal{U})_q[-q]$ with a coboundary operator $\delta_k = \sum_{k = 0}^q (-1)^k \partial_{k, q}$.

\bigskip
\noindent 
There are a few homotopy equivalent ways of defining the homotopy limit of a semicosimplical DGLA.  One is  the Thom-Whitney-Sullivan construction \cite{manetti}.  The resulting object is a DGLA, with a simpler Maurer Cartan equation, but the complex is quite large and harder
to interpret geometrically.   Bandiera introduced a smaller version of the homotopy limit with an $L_\infty$ structure through homotopy transfer.   In the case of interest to us, it is isomorphic to the Cech complex as a DG vector space. 

\begin{definition}
	Given a semicosimplcial complete DGLA $L_\bullet\in\widehat{\mathbf{DGLA}}^{\underrightarrow{\Delta}}$, its \textbf{Thom-Whitney complex} is the complete DGLA
	\[  \operatorname{Tot_{TW}}(L_\bullet)=\left\{ (\alpha_0,\ldots,\alpha_n,\ldots)\in\prod_{n\geq0}\widehat{\Omega}^*(\Delta_n;L_n)\,\,\operatorname{s.t.}\,\,\partial^j_\ast(\alpha_{n-1}) =\delta^*_j(\alpha_n) \right\}  \] 
	where the morphism $\partial^j_\ast:\widehat{\Omega}^*(\Delta_{n-1};L_{n-1})\to \widehat{\Omega}^*(\Delta_{n-1};L_{n})$ is the push-forward by the $j$-th cofaces of $L_\bullet$ and $\delta_j^*:\widehat{\Omega}^*(\Delta_{n};L_{n})\to \widehat{\Omega}^*(\Delta_{n-1};L_{n})$ is the pull back by the $j$-th coface of $\underrightarrow{\Delta}$. The Thom-Whitney complex inherits 
	a DGLA structure from $L_\bullet$ since the product on differential forms
	is graded commutative. 
\end{definition}

\begin{definition}
	Given a semicosimplcial complete DGLA $L_\bullet\in\widehat{\mathbf{DGLA}}^{\underrightarrow{\Delta}}$, its \textbf{totalization} $\operatorname{Tot}(L_\bullet)$ is the complete $L_\infty$ algebra 
	\[  \operatorname{Tot}(L_\bullet)=\left\{ (\alpha_0,\ldots,\alpha_n,\ldots)\in\prod_{n\geq0}C^*(\Delta_n;L_n)\,\,\operatorname{s.t.}\,\,\partial^j_\ast(\alpha_{n-1}) =\delta^*_j(\alpha_n) \right\}  \] 
	where the morphism $\partial^j_\ast:C^*(\Delta_{n-1};L_{n-1})\to C^*(\Delta_{n-1};L_{n})$ is the push-forward by the $j$-th cofaces of $L_\bullet$ and $\delta_j^*:C^*(\Delta_{n};L_{n})\to C^*(\Delta_{n-1};L_{n})$ is the pull back by the $j$-th coface of $\underrightarrow{\Delta}$. The $L_\infty$-structure on 
	 $\operatorname{Tot}(L_\bullet)$ is induced from the DGLA structure of
	 $\operatorname{Tot_{TW}}(L_\bullet)$ via the contraction induced
	 by the contraction of differential forms $\Delta_n$ onto cochains on 
	 $\Delta_n$ (with a fixed choice of Dupont homotopy, cf. \cite{getzler}.) 
\end{definition}

\begin{proposition}
	In $L_\bullet$ is a semicosimplicial Lie algebra (i.e. all $L_k$ are in 
	homological degree zero) then underlying dg vector space of $\Tot(L_\bullet)$ is the complex
	$L_0 \to L_1 \to L_2 \ldots$ with the differential $\sum_j (-1)^j \partial_{j, q}$. 
\end{proposition}

\begin{proof}
	A degree $m$ element $\alpha$ in $\operatorname{Tot}(L_\bullet)$ is of the form $\alpha = (\alpha_0, \dots, \alpha_n, \dots)$, where $\alpha_n$ is 
	 in $C^m(\Delta_n; L_n)$ (with additional compatibility 
	conditions on $\alpha_n$). Since $L_m$ is in homological degree zero, 
	$\alpha_m$ only takes nonzero value on (nondegenerate) $m$-simplicies of $\Delta_n$. In particular $\alpha_n = 0$ for $n < m$. 
	
	  For $n >  m$, by definiton of $Tot$ evaluation of $\alpha_n$ on $m$-simplices of $\Delta_n$ is given by restricting to an $(n-1)$ dimensional 
	  face containing and $m$-simplex and then computing the value by applying
	  a cosimplicial map to the value of $\alpha_{n-1}$. Thus, a degree $m$ cochain $\alpha$
	  is uniquely determined by $\alpha_m$.  
	  
	  An explicit check 
	  also shows that this correspondence is compatible with differentials: if $l$ is  an element  in $L_k$, then $d(l) = \sum^{k+1}_{j=0} (-1)^j \partial_{j, k+1} (l) \in L_{k+1}$.  Now consider $\alpha$ an element of degree $k$ in $\operatorname{Tot}(L_\bullet)$ whose value on the $k$- simplices in $\Delta_k$ is $l$.  $\alpha$ is of the form $\alpha = (0, \dots, 0, \alpha_k, \alpha_{k+1}, \dots)$ and $d(\alpha) = (\delta^*(0), \dots, \delta^*(0), \delta^*(\alpha_k), \delta^*(\alpha_{k+1}), \dots)$.
	
	Note that $\delta^*(\alpha_i) = \partial_*(\alpha_{i-1})$ by the construction of $\operatorname{Tot}$ and $\alpha_{k-1}$ is the 0 cochain, so $\delta^*(\alpha_k) = 0$.  By the same reasoning we have $\delta^*(\alpha_{k+1}) = \partial_*(\alpha_k)$, whose evaluation at the $k+1$ simplex in $\Delta_{k+1}$ is $\sum^{k+1}_{j=0} (-1)^j \partial_j (l) = d(l)$, so we have $d(\alpha) = (0, \dots, 0, 0, \delta^*(\alpha_{k+1})=d(l), \dots)$.  Thus by our previous discussion, $d(\alpha)$ must be a degree $k+1$ element in $\operatorname{Tot}(L_\bullet)$ which under our bijection will precisely be $d(l)$ in $\displaystyle \bigoplus_nL_n[-n]$.
\end{proof}
\bigskip
\noindent
\textbf{Notation.} In the geometric situation with the sheaf of Lie algebras $\g$
and an open covering $\mathcal{U}$ we will denote $Tot(\mathfrak{g}(\mathcal{U})_\bullet)$ by 
$\mathcal{L}(\mathfrak{g})$ and $\TW(\mathfrak{g}(\mathcal{U})_\bullet) $ by 
$\widehat{\mathcal{L}}(\g)$. We will also use $[-,-]_{\g}$ to denote $[-,-]_{\TW(\g)}$. By the previous result, $\mathcal{L}(\mathfrak{g})$ is an $L_\infty$ algebra 
for which the underlying vector space is just the Cech complex of $\mathfrak{g}$
with respect to $\mathcal{U}$. It also depends on the choice of Dupont homotopy 
but with suppress both dependences (on the covering and on the homotopy) 
from notation, assuming that both are fixed. 

\subsection{Theorems on Descent of Deligne Groupoids}
Now we  state Hinich's theorem on descent of Deligne groupoid.  We will also 
sketch a proof of  Fiorenza, Manetti, and Martinengo's result that  for a semicosimplical DGLA  in degree 0, we get an \textit{isomorphism} of groupoids.



\begin{theorem}[Hinich, \cite{hinich}]
	For semicosimplicial DGLAs $L_\bullet$ concentrated in non negative degrees, the Deligne functor commutes with homotopy limits, i.e., there is a natural equivalence of groupoids
	\[\operatorname{Del}(\operatorname{Tot}(L_\bullet)) \simeq \operatorname{Tot}(\operatorname{Del}(L_\bullet)).\]
\end{theorem}

$\operatorname{Tot}(\operatorname{Del}(L_\bullet))$ (the left hand side) is called the groupoid of descent data on $L_\bullet$.  In the case where $L_\bullet$ is a secosimplicial Lie algebra, its objects are the nonabelian 1-cocycles
\[Z^1(\exp(L_1)) = \{m \in L_1|e^{\partial_0(m)}e^{-\partial_1(m)}e^{\partial_2(m)} = 1\}\]
and its morphisms between two cocycles $m_0$ and $m_1$ are
\[\{a \in L_0|e^{-\partial_1(a)}e^{m_1}e^{\partial_0(a)} = e^{m_0}\}\]
We will give details on the above theorem but rather move to the case when $L_\bullet$
is formed by Lie algebras in degree zero (as it happens for the geometric situation involving an
open cover and a sheaf of Lie algebras). Then Fiorenza, Manetti, and Martinengo proved in \cite{manetti} that instead of just equivalence, we are getting an \textit{isomorphism} of groupoids, that is, the nonabelian 1-cocyles, as a subset of $\Tot(L_\bullet)$, are the same as solutions of the $L_\infty$ Maurer Cartan equation on $\Tot(L_\bullet)$, and that two nonabelian cocyles are equivalent iff they are equivalent Maurer Cartan elements.
We sketch the proof briefly, focusing on the parts which will be needed later for 
studying torsors. 

\begin{theorem} [Fiorenza, Manetti, and Martinengo, \cite{manetti}]
	\label{importanttheorem}
	For semicosimplicial Lie algebra $L_\bullet$, there is an isomorphism of $\infty$-groupoids
	\[\Delf(\operatorname{Tot}(L_\bullet)) \cong \operatorname{Tot}(\Delf(L_\bullet)).\]
	and thus an isomorphism of groupoids
	\[\Del(\operatorname{Tot}(L_\bullet)) \cong \operatorname{Tot}(\Del(L_\bullet)).\]
\end{theorem}

First, we should define the totalization of the semicosimplical simplicial set $\operatorname{Del}_\infty(L_\bullet)$ (the right hand side of our isomorphism).  The totalization of semicosimplical simplicial sets is defined the same way as the totalization of semicosimplicial complete $L_\infty$ algebras by simply replacing $C^*(\Delta_i; L_i)$ with $\underline{\operatorname{SSet}}(\Delta_i, L_i)$.  This totalization anjots a universal property 
similar to the totalization of semicosimplicial DGLA.  Using this definition, we have
\[\operatorname{Tot}(\operatorname{Del}_\infty(L_\bullet)) = \left\{ (\alpha_n) \in  \prod_{n \geq 0} \underline{\operatorname{SSet}}(\Delta_n, \operatorname{Del}_\infty(L_n)) \,\,\vert\,\,\partial^j_\ast(\alpha_{n-1}) =\delta^*_j(\alpha_n)\right\}.\]
Notice that $\underline{\operatorname{SSet}}(X, Y)_n = \operatorname{SSet}(\Delta_n \times X, Y)$, so 
$$\operatorname{Tot}(\operatorname{Del}_\infty(L_\bullet))_i = \{ (\alpha_n) \in  \prod_{n \geq 0} \operatorname{SSet}(\Delta_i \times \Delta_n, \operatorname{Del}_\infty(L_n)) \,\,\vert\,\,\partial^j_\ast(\alpha_{n-1}) =\delta^*_j(\alpha_n)\}$$

\begin{proof}
	Since $\Tot(L_\bullet)$ is concentrated in non negative degrees, by Proposition \ref{groupoid-properties} $\Delf(\Tot(L_\bullet)) = \mathcal{N}(\Del^{\op{op}}(\Tot(L_\bullet)))$.  Hence $\Delf(\Tot(L_\bullet))$ is uniquely determined by its 0-simplices, 1-simplices and 2-simplices.  On the other hand, since $L_i$'s are all concentrated in degree 0, $\Delf(L_i) = \mathcal{N}(\Del^{\op{op}}(L_i))$.  
	It is easy to check from definition that $Tot$ and $\mathcal{N}$ commute in this case.
	Thus, $\Delf(Tot(L_\bullet))$ is also uniquely determined by the 0,1,2-simplices and same holds for its homotopy limit $\Tot(\Delf(L_\bullet))$.
	
	Thus we will be comparing the 0,1,2-simplices for $\operatorname{Del}_\infty(\operatorname{Tot}(L_\bullet))$ and $\operatorname{Tot}(\operatorname{Del}_\infty(L_\bullet))$.  
	Using results of the previous sections to untangle definitions (we omit 
	the straightforward computational details), we see that on both 
	sides simplices have identical descriptions. We record them for future use.

	\bigskip
	\noindent
	\textbf{0-simplices} of both  $\operatorname{Del}_\infty(\operatorname{Tot}(L_\bullet))$ and 
	$\operatorname{Tot}(\operatorname{Del}_\infty(L_\bullet))$ can be
	identified with $\alpha_1 \in C^1(\Delta_1, L_1) = L_1$ such that 
	$\partial^0_\ast \alpha_1 \circ(- \partial^1_\ast \alpha_1) \circ \partial^2_\ast \alpha_1 = 0,$
	i.e with nonabelian 1-cocycles in $\Tot(L_\bullet)$ ($\circ$ denotes the CBH product).

	\bigskip
	\noindent
	\textbf{1-simplices}
	of both $\operatorname{Del}_\infty(\operatorname{Tot}(L_\bullet))$ and 
	$\operatorname{Tot}(\operatorname{Del}_\infty(L_\bullet))$ are in bijection with the set of $\alpha_1, \alpha'_1 \in L_1$ and $l_0 \in L_0$ such that the diagram below commutes:
	\[
	\xymatrix{
		\partial_\ast^0(0)\ar[rrr]^{\alpha_1}\ar[ddd]_{\partial_\ast^0(l_0)}& & &\partial_\ast^1(0) \ar[ddd]^{\partial_\ast^1(l_0)}\\ \\ \\ 
		\partial_\ast^0(0)\ar[rrr]^{\alpha'_1}& & &\partial_\ast^1(0)
	}
	\]

\bigskip
\noindent
	\textbf{2-simplices}
	 of $\operatorname{Del}_\infty(\operatorname{Tot}(L_\bullet))$ 
	  of the following form:
	\[ \xymatrix{ & & a\ar[rrdd]^{l_1} & &\\ & & & &\\ l_0 \cdot a\ar[uurr]^{l_0} \ar[rrrr]_{l_2 = l_0 \circ l_1}& & & & -l_1 \cdot a}\]
	where $a$ is a 1-cocycle of the $\Tot((L_\bullet)$, $l_0 \cdot a$ and $-l_1 \cdot a$ are the resulting 1-cocycles when $l_0$ and $-l_1$ act on $a$, $l_0, l_1 \in L_0$ and their composition is given by the Baker-Campbell-Hausdorff formula in $L_0$.
	
	\medskip
	\noindent
	We proved an isomorphism of $\infty$-groupoids 
	$$\Delf(\operatorname{Tot}(L_\bullet)) \cong \operatorname{Tot}(\Delf(L_\bullet))
	$$
	Applying $\pi_{\leq 1}(\ldots)^{op}$  gives 
	$\Del(\operatorname{Tot}(L_\bullet)) \cong \operatorname{Tot}(\Del(L_\bullet))$.
\end{proof}

\subsection{Unipotent torsors and deformations.}

We return to the setting of our Introduction, with an algebraic scheme $X$, its
affine open cover $\mathcal{U} = \{U_i\}$ and a $G$-torsor $P$ give by the transition 
functions  $\Phi_{ij}: U_i \cap U_j \to G$ viewed as exponents of $\varphi_{ij}: U_i \cap U_j \to \g$.
The usual cocycle condition 
$\Phi_{ij} \Phi_{jk} = \Phi_{ik}$ can be understood as a condition imposed on the 
element $\varphi = (\varphi_{ij}) \in \mathcal{L}(\g)^1$ in the degree 1 component of the 
Cech $L_\infty$ algebra of the sheaf of $\g$-valued functions on $X$. 

Similarly, a change of trivialization of $P$ on each $U_i$ is give by a collection of
regular maps $\Sigma_i: U_i \to G$, which changes the cocycle as follows: 
\[\Phi_{ij} \mapsto \Sigma_{i}^{-1} \Phi_{ij} \Sigma_{j}\]
Define the groupoid  $Tors(X, P, \mathcal{U})$ of $P$-torsors on $X$
 (with respect to the fixed choice of an affine cover)
by viewing cocycles $\Phi_{ij}$ as objects and changes of trivializations $\Sigma_j$ as 
morphisms. Composition of morphisms  is the obvious product of $G$ valued cocycles
(which can be rephrased in terms of CBH product on elements $\sigma = (\sigma_j) 
\in \mathcal{L}(\g)^0$). Comparing this with the descriptions of 0, 1, and 2-simplices in the
proof of \ref{importanttheorem} we obtain a
\begin{corollary}
There exists an isomorphism the Deligne groupoid of the Cech $L_\infty$-algebra 
of the sheaf of $\g$-valued regular functions, and $Tors(X, G, \mathcal{U})$.
\end{corollary}

\bigskip
\noindent
A version of this statement can be formulated for the problem of deforming a fixed
$G$-torsor $P$ over an Artinian $k$-algebra $A = k \oplus m_A$ (with finite dimensional
$m_A$, for the sake of simplicity). Then we 
fix a covering $\mathcal{U}$, a cocycle giving $P$ and reduce the question of 
constructing a deformation to the question of finding an element 
$\phi \in \mathcal{L}(\g_P \otimes m_A)$ satisfying appropriate cocycle 
condition. Here $\g_P$ is the adjoint sheaf of Lie algebras, associated to $P$. 
As our further discussion will be a generalization of this picture, we
leave the detailed discussion until later, just stating here the
\begin{corollary}
	There is an isomorphism between the Deligne groupoid of the 
	Cech $L_\infty$ algebra of $\g_P \otimes m_A$ and the groupoid of 
	deformations of $P$ over the spectrum of the Artinian algebra $A = k \oplus m_A$. 
\end{corollary} 
In this form the statement can be extended to a wider range of groups $G$ as 
tensoring with $m_A$ automatically creates a sheaf of nilpotent Lie algebras.

\section{Lifting $G$-torsors across extensions.}

Now we finally move to the problem discussed in the introduction. Suppose we
have an extension of nilpotent Lie algebras: 
\[0 \to \h \to \gh \to \g \to 0\]
corresponding to the extension of unipotent groups $1 \to H \to \tilde{G} \to G \to 1$. 
For a $G$-torsor $P$ on a scheme $X$ we want to study its different lifts $\tilde{P}$ to 
a $\tilde{G}$-torsor. 
To rigidify the problem, fix an open cover of $X$ and a cocycle $\Phi_{ij}: U_i \cap U_j 
\to G$ defining $P$, and study the groupoid of $\tilde{G}$-valued cocycles 
$\tilde{\Phi_{ij}}: U_i \cap U_j 
\to \tilde{G}$ lifting $\Phi_{ij}$. Morphisms between two such lifts are given by 
change of coordinate maps $\tilde{\Sigma}_j: U_j \to \tilde{G}$ which map to identity 
in $G$. This means that $\tilde{\Sigma}_j$  take
values in the subgroup $H \subset \tilde{G}$. 
We will show that our choices induce on the Cech complex $\mathcal{L}(\h)$ of 
$\h$-valued functions the structure of a \textit{curved} $L_\infty$-algebra and that the 
groupoid of lifts of $P$ is isomorphic to the (appropriately defined!) Deligne groupoid
of $\mathcal{L}(\h)$.

\subsection{Lie Algebra Extensions}
\begin{definition}
	Let $\g$ and $\h$ be two Lie algebras.  An \textbf{extension} $\gh$ of $\g$ by $\h$ is a short exact sequence of the form
	\[0 \to \h \to \gh \to \g \to 0.\]
\end{definition}

\begin{definition}
	Let $\gh$ and $\gh'$ be two extensions of $\g$ by $\h$.  $\gh$ and $\gh'$ are said to be equivalent if there exits a commutative diagram
	\[\begin{tikzcd}
		0 \arrow[r] & \h \arrow[r] \arrow[d, equal] & \gh \arrow[r] \arrow[d, "\varphi"] & \g \arrow[r] \arrow[d, equal] & 0\\
		0 \arrow[r] & \h \arrow[r] & \gh' \arrow[r] & \g \arrow[r] & 0
	\end{tikzcd}\]
\end{definition}

\begin{definition}
	A \textbf{non-abelian 2-cocycle} on $\g$ with values in $\h$ is a couple $(c, b)$ of linear maps $c: \g \wedge \g \to \h$ and $b: \g \to \op{Der}(h)$ satisfying 
	\[[b(x), b(y)] - b([x,y]) = \op{ad}(c(x,y))\]
	and
	\[\displaystyle \sum b(x, c(y,z)) - c(b(x,y), z) = 0\]
	where the sum is over cyclic permutations of $x$, $y$, and $z$.   Two non-abelian 2-cocyles are equivalent, $(c,b) \sim (c',b')$ if there exists $\beta: \g \to \h$ satisfying 
	\[b'_x = b_x + \op{ad}_{\beta(x)}\]
	and
	\[c'(x,y) = c(x,y) + b_x(\beta(y)) - b_y(\beta(x)) - \beta([x,y]) + [\beta(x), \beta(y)]\]
	\end{definition}

Choosing a vector space splitting $\tilde{\g} = \g \oplus \h$ compatible with the embedding of $\h$ and projection to $\g$, and writing out the bracket of $\gh$ gives such a non-abelian 
cocycle.  For $x \in \g$ and $y \in \h$, $b: \g \to \op{Der}(\h)$ is given by $b(x)(y)$
equal to the projection of $[x, y]_{\tilde{\g}}$ onto the $\h$-component,  and for $x,x' \in \g$, $c: \g \wedge \g \to \h$ is given by the $\h$ component of $[x,x']_{\gh}$.
A direct computation shows that extensions of $\g$ by $\h$ are classified by equivalence classes of non-abelian cocycles, which can further be reformulated as 
equivalence classes of Maurer Cartan solutions for some DGLA, see \cite{lieext}.
For our purpose, it is enough for us to know that the maps $b$ and $c$ fully describe an extension.

\subsection{Twisted Cocycles and their Equivalence}

We can identify $\tilde{G}$ with the product $G \times H$ as varieties (not groups!) 
by choosing a vector space
splitting of $\gh =  \g \oplus \h$ and thus get an embedding $G \xhookrightarrow{} \tilde{G}$ (of varieties!) by embedding $\g \xhookrightarrow{} \gh$ and using the exponential map.   Multiplication in $\tilde{G}$ is not the regular multiplication of $G \times H$ and is determined by $c$ and $b$ (via the CBH formula).
Our question is then what are the principal $\tilde{G}$-bundles $\tilde{P}$ that extend $P$.
The splitting $\tilde{G} = G \times H$  allows us to rewrite the unknown lifted
cocycle as 
$$
\tilde{\Phi}_{ij} = \Phi_{ij} \Psi_{ij} = \exp(\varphi_{ij})\exp(\psi_{ij}); \quad 
\varphi_{ij} \in \Gamma(U_{ij}, \g), \psi_{ij} \in \Gamma(U_{ij}, \h)
$$
The cocycle condition can then be rewritten as (product is taken in $\tilde{G}$)
$$
	\exp(\varphi_{ij})\exp(\psi_{ij})\exp(\varphi_{jk})\exp(\psi_{jk})
	= \exp(\varphi_{ik})\exp(\psi_{ik})
$$
Comparing with the product of  $\{exp(\varphi_{ab)}\}$ in $G$ we get
$$
	\exp(\varphi_{ij})\cdot_{\tilde{G}}\exp(\varphi_{jk})
	= (\exp(\varphi_{ij})\cdot_{G}\exp(\varphi_{jk}))\mathcal{C}(\varphi_{ij},\varphi_{jk})
	= \exp(\varphi_{ik})\mathcal{C}(\varphi_{ij},\varphi_{jk})
$$
where $\mathcal{C}(\varphi_{ij},\varphi_{jk})$ is the $H$ component of $\exp(\varphi_{ij})\cdot_{\tilde{G}}\exp(\varphi_{jk})$.  If we rewrite $\exp(\varphi_{ij})\cdot_{\tilde{G}}\exp(\varphi_{jk})$ in Lie algebra terms using the Baker-Campbell-Hausdorff formula on $\gh$, i.e. $\varphi_{ij} \ast_\gh \varphi_{jk}$, then $\mathcal{C}(\varphi_{ij},\varphi_{jk})$ is precisely the exponent of the $\h$ component of $\varphi_{ij} \ast_\gh \varphi_{jk}$.
Combining this with the fact that 
\begin{align*}
	\exp(-\varphi_{jk})\exp(\psi_{ij})\exp(\varphi_{jk}) = \exp(\sum^\infty_{s=0} (-1)^s \frac{(\op{ad}_{\varphi_{jk}})^s}{s!} (\psi_{ij})),
\end{align*}
and denoting 
\begin{equation}
 exp(\psi) ^{\varphi}= \exp\bigg(\sum^\infty_{s=0} \frac{(\op{ad}_{\varphi})^s}{s!} (\psi)
 \bigg)
\end{equation}
the twisted group cocycles extending  $P$ can then we rewritten as
\begin{equation}
		\label{twisted-cocycle}
\mathcal{C}(\varphi_{ij},\varphi_{jk}) exp(\psi_{ij})^{-\varphi_{jk}}
\exp(\psi_{jk}) = \exp(\psi_{ik}).
\end{equation}
Passing to a different lift of the same cocycle for $P$ corresponds to
\[\tilde{\Phi}_{ij} \mapsto \Sigma_{i}^{-1} \tilde{\Phi}_{ij} \Sigma_{j}
= \Phi_{ij} (\Phi_{ij}^{-1} \Sigma_i^{-1} \Phi_{ij}) \Psi_{ij} \Sigma_j
\]
where $\Sigma_{a} \in \Gamma(U_{a}, H)$.  
Writing $\Sigma_a = \exp(\sigma_a)$ and comparing the $H$ components 
written in Lie algebra terms we get that two twisted cocyles $\{\exp(\psi_{ij})\}$ and $\{\exp(\psi_{ij}')\}$ are equivalent iff there exist $\{\sigma_{a} \in \Gamma(U_{a}, \h)\}$ such that 
\begin{equation}
		\label{twisted-equivalence}
	\exp(\psi_{ij}') = \exp(-\sigma_{i})^{-\varphi_{ij}} \exp(\psi_{ij}) \exp(\sigma_j)
\end{equation} 
on all double overlaps $U_i \cap U_j$.

Thus, lifting  the cocycle of $P$ is the same as  lifting a Maurer-Cartan 
element $a$ from $\mathcal{L}(\g)$ to $\mathcal{L}(\gh)$.  Lifting a Maurer Cartan solution over a surjection is difficult in general but we use a vector space splitting $\mathcal{L}(\gh) = \mathcal{L}(\h) \oplus \mathcal{L}(\g)$ and to define a
curved $L_\infty$ algebra structure on $\mathcal{L}(\h)$.
Our goal is to show that the twisted cocycle condition (\ref{twisted-cocycle})
is the same as the curved Maurer Cartan equation on $\mathcal{L}(\h)$ and the equivalence of twisted cocycles (\ref{twisted-equivalence}) 
is the same as the equivalence of curved Maurer Cartan solutions on $\mathcal{L}(\h)$.

\section{Maurer Cartan Solutions for Thom Whitney}
\subsection{Bijection between Maurer Cartan Solutions}

Recall that he Thom-Whitney complex $\hat{\mathcal{L}}(\gh)$ of $\ghu$ has the 
structure of a DGLA .  The vector space splitting $\gh = \g \oplus \h$ induces a
splitting of complexes $\hat{\mathcal{L}}(\gh) = \hat{\mathcal{L}}(\g) \oplus \hat{\mathcal{L}}(\h)$. 
Now fix  a Maurer-Cartan solution $a \in \hat{\mathcal{L}}(\g)^1$,  and 
consider $\alpha \in \hat{\mathcal{L}}(\h)^1$. Then the 
Maurer-Cartan equation for $a + \alpha$
$$
d(a+\alpha) + \frac{1}{2}[a+\alpha, a+\alpha]_{\gh}=0
$$
reduces to 
$$
 \frac{1}{2} c(a,a) + (d + \ada)(\alpha) + \frac{1}{2}[\alpha,\alpha]_{\h} = 0
 $$
 This can be viewed as a curved Maurer-Cartan equation with respect to the 
 new (curved) differential $d_{\h}$ on $\hat{\mathcal{L}}(\h)$ given by 
 the restriction of  $d + \ada$. Setting  $C = \frac{1}{2} c(a,a)$, we obtain by 
 a straightforward computation that $d_\h^2 = [C,-]$. This means that 
 $(\hat{\mathcal{L}}(\h), C, d_{\h}, [-,-]_{\h})$ is a curved DGLA (see Appendix
 for general theory). 

For a fixed Maurer-Cartan element $a \in \hat{\mathcal{L}}(\g)$, let 
$\MC_a(\hat{\mathcal{L}}(\gh))$ be the set of Maurer Cartan solutions in 
$\hat{\mathcal{L}}(\gh)$ which have the form $a + \alpha$ with 
$\alpha \in \hat{\mathcal{L}}(\h)$. By the above discussion we have a
bijection of this set 
\[\MC_a(\hat{\mathcal{L}}(\gh)) \cong \MC(\hat{\mathcal{L}}(\h))\]
with the set $\MC(\hat{\mathcal{L}}(\h))$  of curved Maurer Cartan solutions 
for $\hat{\mathcal{L}}(\h)$.

\subsection{Equivalences of Maurer Cartan Solutions}
Recall that  an equivalence of two Maurer Cartan solutions $z, z' \in \MC(\hat{\mathcal{L}}(\gh))$ (a morphism of the corresponding groupoid) 
is an element  
$$
\tilde{z}\in \MC(\hat{\mathcal{L}}(\gh) \otimes k[s,ds]) \qquad 
\textrm{ such that }
\tilde{z}|_{s=0} = z, \qquad  \tilde{z}|_{s=1} = z'.
$$
Recall that $k[s, ds]$ is the graded commutative algebra of polynomial forms on a 
line and the tensor product $\hat{\mathcal{L}}(\gh) \otimes k[s,ds]$
has the induced DGLA structure. 
The evaluation map $\op{Eval}_{s = s_0}$ is given in Section 2.1. 

\bigskip
\noindent 
We will check that an equivalence  between Maurer-Cartan solutions
 $a+\alpha$ and $a+\alpha'$, which is constant in the first component, 
  essentially amounts to an equivalence of curved Maurer-Cartan 
 solutions $\alpha$ and $\alpha'$. 

\bigskip
\noindent
Notice that $\tilde{z} \in \hat{\mathcal{L}}(\gh) \otimes k[s,ds]$ can be written as $\tilde{a} + \tilde{\alpha}$ where $\tilde{a} \in \hat{\mathcal{L}}(\g) \otimes k[s,ds]$ and $\tilde{\alpha} \in \hat{\mathcal{L}}(\h) \otimes k[s,ds]$.

But since we are not changing the trivialization of $P$, we are only interested in the case
where  $\tilde{a} \in \hat{\mathcal{L}}(\g) \otimes k[s,ds]$ is constant
and equal to $a$. 

A straightforward check shows that  $(\hat{\mathcal{L}}(\h) \otimes k[s,ds], \tilde{d}_\h, [-,-]_{\hat{\h}})$ does in fact have a curved DGLA structure and
 that  the curved Maurer Cartan equation for $\alpha \in \hat{\mathcal{L}}(\h) \otimes k[s,ds]$ is exactly the same as the Maurer Cartan equation for $a + \alpha \in \hat{\mathcal{L}}(\gh) \otimes k[s,ds]$.

Thus, if $a + \tilde{\alpha}(s, ds)$ gives a homotopy equivalence between Maurer-Cartan 
solutions  $a+ \alpha$ and $a+ \alpha'$ then $\tilde{\alpha}(s, ds)$ gives a curved
homotopy equivalence between curved Maurer-Cartan solutions $\alpha$ and 
$\alpha'$  in $\hat{\mathcal{L}}(\h)$.

Recall that for an $L_\infty$ algebra $L$ we have a contraction 
\[
\begin{tikzcd}
	C^*(\Delta_1; L) \arrow[r,shift left] \arrow[r, leftarrow, shift right] & L \otimes \Omega_1 \arrow[loop right, leftarrow]{l}{K}
\end{tikzcd}
\]
where $K$ is the Dupont homotopy, and that application of formal Kuranishi theorem 
identifies 1-simplices of $\Delf(L)$ with the set
\[\MC(L \otimes \Omega_1, K) = \{x \in \MC(L \otimes \Omega_1)|K(x)=0\}\]
Now define $\MC_a(\hat{\mathcal{L}}(\gh) \otimes \Omega_1, K)$ as the subset of $\MC(\hat{\mathcal{L}}(\gh) \otimes \Omega_1, K)$ where elements are of the form $a + \tilde{\alpha}(s, ds)$.   Since $K(a) = 0$, $K(a + \tilde{\alpha}) = 0$ iff $K(\tilde{\alpha})=0$.  We have a bijection 
\[\MC_a(\hat{\mathcal{L}}(\gh) \otimes \Omega_1, K) \cong \MC(\hat{\mathcal{L}}(\h) \otimes \Omega_1, K).\]
(where similarly, the right hand side is defined as curved Maurer-Cartan solutions
annihilated by 
A similar argument with $\Omega_2$ shows that compositions of morphisms agree too;, i.e. if $(a + \tilde{\alpha} )\circ (a + \tilde{\alpha}' )= a + \tilde{\alpha}''$, then $\tilde{\alpha} \circ \tilde{\alpha}' = \tilde{\alpha}''$.

Denote $\Del_a(\hat{\mathcal{L}}(\gh))$ the groupoid with objects $\MC_a(\hat{\mathcal{L}}(\gh))$ and morphisms $\MC_a(\hat{\mathcal{L}}(\gh) \otimes \Omega_1, K)$, $\Del_a(\hat{\mathcal{L}}(\gh))$ is then a subgroupoid (i.e. 
subcategory) of $\Del(\hat{\mathcal{L}}(\gh))$).  The identification of $a + \alpha$ with $\alpha$ clearly gives us a morphism of groupoids, so we have an isomorphism of groupoids 
\[\Del_a(\hat{\mathcal{L}}(\gh)) \cong \Del(\hat{\mathcal{L}}(\h)).\]

\begin{remark}
	In general if $L$ is a curved $L_\infty$ algebra, $\Del(L)$ might not be well defined.   However, in our situation, if we define $\Del(\hat{\mathcal{L}}(\h))$, by mimicking $\Del$ in the non curved case, as the groupoid whose objects are $\MC(\hat{\mathcal{L}}(\h))$ and whose (opposite) morphisms $\MC(\hat{\mathcal{L}}(\h) \otimes \Omega_1, K)$, we will get a groupoid isomorphic to the subgroupoid of $\Del(\hat{\mathcal{L}}(\gh))$.  Furthermore, we can define $\Delf(\hat{\mathcal{L}}(\h))$ as $\mathcal{N}(\Del^{\op{op}}(\hat{\mathcal{L}}(\h))$.  
\end{remark}

\section{Extensions of Cocycles and Curved $L_\infty$ Maurer Cartan Solutions}

We can now put every piece together by using formal Kuranishi to relate $\MC(\hat{\mathcal{L}}(\gh))$ to the cocycles for extensions of principal $G$-bundles $P$ and relate $\MC(\hat{\mathcal{L}}(\h))$ to the Maurer Cartan set of the curved $L_\infty$ algebra $\mathcal{L}(\h)$. If $P$ is trivialized on the elements of an affine cover $\mathcal{U}$, the transition functions of $P$ give us a Maurer Cartan element $a \in \MC(\mathcal{L}(\g))$.  By the formal Kuranishi theorem, this lifts to a 
unique Maurer-Cartan solution of the Thom- Whitney algebra $\hat{\mathcal{L}}(\g))$
which is annihilated by the Dupont contraction.  To simplify the notation, we will denote both elements $a$.

Our goal is to prove a version of the groupoid isomorphism in the previous section, but
replacing Thom-Whitney DGLAs with Cech $L_\infty$-algebras. Thus we denote by denote $\MC_a(\mathcal{L}(\gh)) \subset \mathcal{L}(\gh)$ the set of Maurer Cartan solutions for the 
form $a+\alpha$, $\Del_a(\mathcal{L}(\gh))$ the groupoid with objects $\MC_a(\mathcal{L}(\gh))$ and the morphisms being those 
morphisms in the Deligne groupoid $\Del(\mathcal{L}(\gh))$ which project on 
the identity of $a$ in  $\Del(\mathcal{L}(\g))$ .

\begin{theorem}
	\label{maintheorem}
	There is an isomorphism of groupoids
	\[\Del_a(\mathcal{L}(\gh)) \cong \Del(\mathcal{L}(\h))\]
	where $\mathcal{L}(\h)$ has a curved $L_\infty$ structure 
	obtained by contraction from the curved DGLA $\hat{\mathcal{L}}(\h)$ whose differential is given by $d + \ada$ and curvature $\frac{1}{2}c(a,a)$ (see Appendix).
\end{theorem}

Before we start the proof, recall that 
 $L_\infty$ structure on  $\mathcal{L}(\f)$ is obtained by homotopy transfer from $\hat{\mathcal{L}}(\f)$ with homotopy $\tilde{K}$ which 
  is termwise the Dupont homotopy \cite{getzler}.
The formal Kuranishi theorem gives us $\MC(\hat{\mathcal{L}}(\f),\tilde{K}) \cong \MC(\mathcal{L}(\f))$
where $\MC(\hat{\mathcal{L}}(\f),\tilde{K}) = \{z \in \MC(\hat{\mathcal{L}}(\f))|\, \tilde{K}(z) = 0\}$.

Similarly, the formal Kuranishi theorem applied to the contraction 
\[
\begin{tikzcd}
	\mathcal{L}(\f)\otimes \Omega_\bullet \arrow[r,shift left] \arrow[r, leftarrow, shift right] &\hat{\mathcal{L}}(\f) \otimes \Omega_\bullet \arrow[loop right, leftarrow]{l}{\tilde{K} \otimes \op{id}}
\end{tikzcd}
\]
gives a simplicial isomorphism 
$\MC(\hat{\mathcal{L}}(\f) \otimes \Omega_\bullet , \tilde{K}\otimes \id) \cong \MC(\mathcal{L}(\f)\otimes \Omega_\bullet)$. 
Note that curved homotopy transfer of structure theorem and thus curved formal Kuranishi theorem only apply when we have a correct filtration on our complexes.  We will discuss this in more detail as a remark when we apply the curved homotopy transfer of structure theorem in the proof.

\begin{proof}
	Define $\MC_a(\hat{\mathcal{L}}(\gh), \tilde{K})$ as the set of Maurer Cartan solutions $a+\alpha \in \MC(\hat{\mathcal{L}}(\gh))$, where  $\alpha \in \hat{\mathcal{L}}(\h)$, such that $\tilde{K}(a + \alpha) = 0$.  By the remarks before the proof and the 
	previous section  we have
	\[\MC_a(\mathcal{L}(\gh)) \cong 
	\MC_a(\hat{\mathcal{L}}(\gh), \tilde{K}) \cong \MC(\hat{\mathcal{L}}(\h),\tilde{K})
	 \cong \MC(\mathcal{L}(\h))
	\]
	where the last bijection follows by applying  the curved  formal Kuranishi theorem
	(see  \cite{fukaya}, \cite{getzlercurved} and Section 8.2 below) to the contraction
	\[
	\begin{tikzcd}
		\hat{\mathcal{L}}(\h) \arrow[r,shift left] \arrow[r, leftarrow, shift right] &\mathcal{L}(\h) \arrow[loop right, leftarrow]{l}{\tilde{K}}
	\end{tikzcd}
	\]
	Thus, the Maurer Cartan solutions are in bijection.
	
	For the arrows (equivalences), we consider two contractions 
	\[
	\begin{tikzcd}
		C^*(\Delta_1; \hat{\mathcal{L}}(\gh)) \arrow[r,shift left] \arrow[r, leftarrow, shift right] & \hat{\mathcal{L}}(\gh) \otimes \Omega_1 \arrow[loop right, leftarrow]{l}{K}
	\end{tikzcd}
	\]
	\[
	\begin{tikzcd}
		\mathcal{L}(\gh)\otimes \Omega_1 \arrow[r,shift left] \arrow[r, leftarrow, shift right] &\hat{\mathcal{L}}(\gh)\otimes \Omega_1 \arrow[loop right, leftarrow]{l}{\tilde{K} \otimes \op{id}}
	\end{tikzcd}
	\]
	Notice that $K$ is a contraction that contracts the $\Omega_1$ part without changing the coefficients and $\tilde{K}\otimes \id$ contracts the coefficients without changing the $s, ds \in \Omega_1 = k[s,ds]$. We will denote by $\MC(\ldots, K, \tilde{K}\otimes \id)$
	and $\MC_a(\ldots, K, \tilde{K}\otimes \id)$ subsets of Mauer-Cartan solutions 
	annihilated by both homotopies (we apply this to sets arising from $\gh$ and 
	their curved analogues arising from $\h$). Then
	\[ 	\MC_a(\hat{\mathcal{L}}(\gh) \otimes \Omega_1, \tilde{K}\otimes \id) \cong 
	\MC_a(\mathcal{L}(\gh) \otimes \Omega_1)\]
	and hence 
	\[\MC_a(\mathcal{L}(\gh) \otimes \Omega_1, K)  \cong \MC_a(\hat{\mathcal{L}}(\gh) \otimes \Omega_1, K, \tilde{K}\otimes \id).\]
	The previous section implies that the set on the right can be identified with 
	$ \MC(\hat{\mathcal{L}}(\h) \otimes \Omega_1, K, \tilde{K}\otimes \id)$, which by a
	similar reasoning applied to $\h$ is bijective to $\MC(\mathcal{L}(\h)\otimes \Omega_1, K)$. Thus we get a bijection 
	\[\MC_a(\mathcal{L}(\gh) \otimes \Omega_1, K) \cong \MC(\mathcal{L}(\h)\otimes \Omega_1, K).\]
	
	Now suppose $a + \tilde{\alpha} \in \MC_a(\mathcal{L}(\gh) \otimes \Omega_1, K)$ 
	gives an equivalence of two  Maurer Cartan solutions $a + \alpha$,  $a + \alpha' \in \MC_a(\mathcal{L}(\gh))$.
	Because of the bijections $\MC_a(\mathcal{L}(\gh)) \cong \MC(\mathcal{L}(\h))$ and $\MC_a(\mathcal{L}(\gh) \otimes \Omega_1, K) \cong \MC(\mathcal{L}(\h)\otimes \Omega_1, K)$, we can uniquely lift $a + \alpha$, $a + \alpha'$, and $a + \tilde{\alpha}$ in $\MC(\mathcal{L}(\h))$ and $\MC(\mathcal{L}(\h)\otimes \Omega_1, K)$ respectively.  Note that after the lift we still have $\tilde{\alpha}|_{s=0} = \alpha$ and $\tilde{\alpha}|_{s=1} = \alpha'$.  The lifts agree because Getzler defines them as solutions to differential equations with initial conditions.  Since the bijection between 
	\[\MC_a(\hat{\mathcal{L}}(\gh) \otimes \Omega_1, K) \cong \MC(\hat{\mathcal{L}}(\h) \otimes \Omega_1, K)\]
	respects composition as shown in the last section, and the homotopy transfer 
	also respects composition of morphisms are respected,  we have an isomorphism of groupoids
	$\Del_a(\mathcal{L}(\gh)) \cong \Del(\mathcal{L}(\h))$
\end{proof}

\begin{remark}
	Note that both totalizations of  a semicosimplicial Lie algebra $L_\bullet$, $\TW(L_\bullet)$ and $\Tot(L_\bullet)$, are equipped with the filtrations
	\[F^{-1} \subset F^{0}\subset F^{1}\subset \dots\]
	When $L_\bullet$ arises from a sheaf of Lie algebras and an open cover
	$F^i$ is given by direct sums of terms coming from sections on $\geq (i+2)$ 
	overlapping open sets. These filtrations are complete. 
	
	The curved $L_\infty$ structure on $\hat{\mathcal{L}}(\h)$, in Getzler's sense, must then be in $F^1S^1(\hat{\mathcal{L}}(\h),\hat{\mathcal{L}}(\h))$ as $C$ is an element in the triple intersection, $d$ and $[-,-]$ are degree 1 maps in $\hat{\mathcal{L}}(\h)[1]$.  This means that $\hat{\mathcal{L}}(\h)$ is pro-nilpotent and thus we can apply curved homotopy transfer of structure theorem and curved formal Kuranishi theorem.  See Appendix for details on curved $L_\infty$ algebras.
\end{remark}

\begin{corollary}
	Suppose we have a principal $G$-bundle $P$ over the base space $X$, $\pi: P \to X$, and $G = \exp(\g)$, where $G$ is unipotent and $\g$ is nilpotent.  Suppose we have a Lie algebra extension $\gh$ of $\g$ by another nilpotent Lie algebra $\h$ 
	\[0 \to \h \to \gh \to \g \to 0,\]
	then extensions of the principal $G$-bundle $P$ are given by the curved Maurer Cartan solutions $\MC(\mathcal{L}(\h))$ and the equivalence of extension is precisely the equivalence of curved Maurer Cartan solutions, i.e. the solution for the twisted cocycle condition (\ref{twisted-cocycle}) 
	is in bijection with the curved Maurer Cartan solutions for $\mathcal{L}(\h)$ and the twisted equivalence (\ref{twisted-equivalence})
	is in bijection with the morphisms between curved Maurer Cartan solution for $\mathcal{L}(\h)$ and this bijection respects composition of morphisms (change of trivialization).
\end{corollary}

\begin{proof}
	Result follows directly from Theorem \ref{maintheorem} and \ref{importanttheorem}.
\end{proof}

\begin{example}
	In the case where the image of $c$ is in the center of $\h$, i.e. $c(x,y) \in \op{Z}(\h) \, \forall x,y \in \g$, we will have an honest action of $\g$ ($G$) on $\h$. The extensions of the bundle $P$ are then given by the curved $L_\infty$ Cech complex, $\mathcal{L}(\h)$, of the associated bundle $P_{\h} = (P \times \h)/G$ (which a bundle of Lie algebras) where its curvature is given by $c$.   
\end{example}

\section{Appendix: curved homotopy  transfer}

\subsection{Definitions} 

\begin{definition}
	Let $L$ be a complete graded vector space; a codifferential $Q$ of degree 1 on the symmetric coalgebra $S(L[1]) = \displaystyle \bigoplus_{n \geq 0} \bigodot^n L$ is called a \textbf{curved} $\mathbf{L_\infty}$ \textbf{structure} on $L$.  A curved $L_\infty$ algebra is a complete graded space $(L, F^{\bullet} L, d)$ together with a curved $L_\infty$ structure $Q$ on $L$.
\end{definition}
As in the non curved case, $Q$ is determined by $Q^1: S(L[1]) \to L$.  The maps $q_i = Q^1_i: \bigodot^n L \to L$ give us the (higher) brackets on $L[1]$; $q_0: k \to L$ in particular gives us the \textbf{curvature} element.  The series of equations (general Jacobi identities) given by $Q^2 = 0$ are different from the non curved case as we have to take into account the $q_0$.  In particular, we have
$$
	q_1^2(x) = q_2(q_0,x), \quad 
	q_1(q_0) = 0 \qquad \qquad x\in L.
$$
Thus, $q_1$ is no longer a differential for $L$ and its cohomology is not well defined.  $S(L[1])$ is equipped with a coaugmentation $\eta: k \to S(L[1])$.  When  $Q\eta = 0$, we have $q_0 = 0$ and we recover non curved $L_\infty$ algebras.

\begin{definition}
	A curved $L_\infty$ morphism $F: (L,Q) \to (M,R)$ between curved $L_\infty$ algebras is a morphism $F: S(L[1]) \to S(M[1])$ that commutes with the coproducts, counits, and codifferentials Q and R.
\end{definition}
Again, $F$ is determined by $F^1_i = f_i: \bigodot^i L[1] \to M[1]$ and $F$ can be computed in a fashion similar to the non curved case (but $i \geq 0$).

\begin{definition}
	For a curved $L_\infty$ algebra $(L,Q)$ 
	\[\sum_{n=0}^\infty \frac{1}{n!} q_n(x, \dots, x) = 0 \qquad x\in L^1.\]
	is called the Maurer-Cartan equation. 
   Its solutions form a (possibly empty!)
	\textbf{Maurer-Cartan set}  $MC(L)$ of the curved $L_\infty$ algebra \(L\).
	Two Maurer-Cartan solutions $a,a' \in \MC(L)$ are \textbf{(homotopy) equivalent} if there exist $z \in \MC(L\otimes k[s,ds])$ such that
	$z|_{s=0} = a,  z|_{s=1} = a'$
	where the evaluation map $\op{Eval}_{s = s_0}: L \otimes k[s,ds] \to L$ is given by 
	\[\op{Eval}_{s=s_0}(x(s)+y(s)ds) = x(s_0)\]
\end{definition}

\subsection{Curved Homotopy Transfer and Kuranishi Theorem}
Here we
 state the main result of Getzler's paper \cite{getzlercurved}, the curved version of formal Kuranishi theorem. For complete filtered complexes $L$, $M$ let $S^i(L, M)$ be the
 set of sequences $(a_0, a_1, \ldots)$ where $a_0 \in F_1 M$ and for $n \geq 1$ each 
 $a_n$ is a filtered graded symmetric $n$-linear map from $L$ to $M$ of degree $i$. 
 It carries a filtration with $F_kS^i$ given by multilinear maps that deepen the filtration 
 at least by $k$ steps.

\begin{definition}
	A curved $L_\infty$ algebra $(L, \lambda)$ is \textbf{pro-nilpotent} if $\lambda
	\in F_1S^1(L, L)$.
\end{definition}

\begin{definition}
	Given $a \in S^i(L,M)$ and $b \in S^0(K,L)$, define the composition $a \bullet b \in S^i(K,M)$ by
	\begin{align*}
		(a \bullet b)_n(x_1,\dots,x_n) =& 	\sum_{\sigma \in S_n} \op{sgn}(\sigma)	 \times
		\\ \times 
\sum_{k=0}^n \frac{1}{k!} \sum_{n_1 + \dots + n_k = n}  \frac{1}{n_1!\dots n_k!}a_k(b_{n_1} &(x_{\sigma(1)},\dots), \dots, b_{n_k}(\dots, x_{\sigma(n)}))
	\end{align*}
\end{definition}

For $\bullet$ to be well defined, we need  $b_0 \in F^1L$. The following theorem was originally shown by Fukaya and stated in the current form by Getzler:

\begin{theorem}[Fukaya, \cite{fukaya}\cite{getzlercurved}]
	\label{curvedhomotopytransfer}
	Consider a complete contraction of filtered complexes
	$
	\begin{tikzcd}
		M \arrow[r,shift left, "f"] \arrow[r, leftarrow, shift right, "g", swap] & L \arrow[loop right, leftarrow]{l}{h}
	\end{tikzcd}
	$
	with continuous  $f$ and $g$.  Suppose $L$ is equipped with a pro-nilpotent curved $L_\infty$ structure $\lambda$.  Then there is a unique solution in $S^0(M,L)$ of the fixed-point equation
	\[F = f - h\lambda \bullet F.\]
	Furthermore, $\mu = g\lambda \bullet F \in S^1(M,M)$ is a curved $L_\infty$ structure on $M$, and $F$ is a curved $L_\infty$ morphism from $(M,F^\bullet M, d,\mu)$ to $(L, F^\bullet L, \delta, \lambda)$.
\end{theorem}

Pro-nilpotence is needed 
 to make $F \mapsto f - h\lambda \bullet F$ a contraction mapping under the metric 
$d_c(x,y) = \op{inf}\{c^{-k}|\, x-y \in F^kL\}$
where $c \in \mathbb{R}^{> 1}$.  
\begin{theorem}[Getzler, \cite{getzlercurved}]
	Under the same setting as in Theorem \ref{curvedhomotopytransfer}, the morphism $g$ induces a bijection from $\MC(L,h) \to \MC(M)$.
\end{theorem}

\bibliographystyle{plain}
\bibliography{VBandKLCpaper}

\end{document}